\UseRawInputEncoding
\documentclass[12pt]{article}
\usepackage{amsfonts}
\usepackage{CJK}
\usepackage{amsthm}
\usepackage{amsmath}
\usepackage{indentfirst}
\usepackage[top=1em,bottom=1em,left=6em,right=6em,includehead,includefoot]{geometry}
\begin{document}

\setlength{\parindent}{2em}
\newtheorem{Theorem}{Theorem}
\newtheorem{Lemma}{Lemma}
\newtheorem{Definition}{Definition}
\newtheorem{corollery}{Corollery}

\title{\bf A Poincar\'e-Hopf type formula for a pair of vector fields}
\author{Xu Chen \footnote{{\it Email:} xiaorenwu08@163.com. ChongQing, China. \ \ ORCID: https://orcid.org/0000-0002-8889-8040}}
\date{}
\maketitle

\begin{abstract}
For two complex vector bundles admitting a homomorphism between them, a Poincar\'e-Hopf formula for the difference of the Chern character numbers of these two vector bundles with isolated singularities is established by Huitao Feng, Weiping Li and Weiping Zhang. This article extend their reslut about Poincar\'e-Hopf type formula for the difference of the Chern character numbers to the non-isolated singularities. A special connection $\widetilde{\nabla}_{1}^{E}$ be constructed, the Chern character ${\rm ch}(E,\widetilde{\nabla}_{1}^{E})$ is the key role for the non-isolated singularities. As a consequence, a Poincar\'e-Hopf type formula for a pair of vector field with the function $h^{T_{\mathbb{C}}M}(\cdot,\cdot)$ has non-isolated zero points over a closed, oriented smooth manifold of dimension $2n$ is established.

\noindent{{\bf Keyword£º}Chern character; Signature $\mathbb{Z}_2$-graded; Poincar\'e-Hopf type formula; non-isolated zero points}
\end{abstract}

\section{Introduction}
Let $M$ be a closed, oriented smooth manifold of dimension $2n$. Let $T_{\mathbb{C}}M=TM\otimes\mathbb{C}$ be the complexification of $TM$. Let $g^{TM}$ be a the Riemannian metric on $M$, it induces canonically a complex symmetric bilinear form on $T_{\mathbb{C}}M$, denoted by $h^{T_{\mathbb{C}}M}$(cf.[2]). Let any $K\in\Gamma(T_{\mathbb{C}}M)$ be the section of $T_{\mathbb{C}}M$, then $K=\xi+\sqrt{-1}\eta$, where $\xi$ and $\eta$ be vector field, we define
$$h^{T_{\mathbb{C}}M}(K,K)=|\xi|^2_{g^{TM}}-|\eta|^2_{g^{TM}}+2\sqrt{-1}\langle\xi,\eta\rangle_{g^{TM}}$$
Surely, $h^{T_{\mathbb{C}}M}(K,K)$ is a smooth function on $M$, we
denoted the set of zero pionts of this function by $Zero(K)$. The
Euler number of manifold $M$ is denoted by $\chi(M)$. H. Jacobowitz [7] established the following result: if $Zero(K)=\emptyset$, then $\chi(M)=0$.

In the end of [7], Jacobowitz asked a question like that: Is there a counting formula for $\chi(M)$ of Poincar\'e-Hopf type, when $Zero(K)\neq\emptyset$?\par

Huitao Feng, Weiping Li and Weiping Zhang [4] establish a Poincar\'e-Hopf formula for the difference of the Chern character numbers of two
vector bundles with $Zero(K)$ is isolated, and use the formula they get a Poincar\'e-Hopf type formula to the set of $Zero(K)$ consists of a finite number of points on a spin manifold $M$. This result is an answers of the question asked by Jacobowitz in [7].

We [2] establish a Poincar\'e-Hopf type formula for a pair of sections of an oriented real vector bundle of rank $2n$ over a closed, oriented manifold of dimension $2n$, with isolated zero points, which generalized the corresponding result in [4].

In this article, we will extend the Poincar\'e-Hopf type formula for the difference of the Chern character numbers of two complex vector bundles to $Zero(K)$ is non-isolated.

\begin{Theorem}
$$\langle{\rm ch}(E_{+})-{\rm ch}(E_{-}),[M]\rangle=\sum_{X}\langle\frac{{\rm ch}(E_{+})-{\rm ch}(E_{-})}{e(\mathcal{N}_{X})},[X]\rangle$$
\end{Theorem}

By use of the Poincar\'e-Hopf type formula for
the difference of the Chern character numbers of two complex vector bundles with $Zero(K)$ is non-isolated, we establish a Poincar\'e-Hopf type formula for a pair of vector field with the function $h^{T_{\mathbb{C}}M}(\cdot,\cdot)$
has non-isolated zero points over a closed, oriented smooth manifold of dimension $2n$.
\begin{Theorem}
Let $M$ be a closed, oriented smooth manifold of dimension $2n$ and $n\geq2$. Let
$T_{\mathbb{C}}M=TM\otimes\mathbb{C}$ be the complexification of
$TM$. Let $g^{TM}$ be a the Riemannian metric on $M$, it induces
canonically a complex symmetric bilinear form on $T_{\mathbb{C}}M$,
denoted by $h^{T_{\mathbb{C}}M}$. Let $K=\xi+\sqrt{-1}\eta$, where
$\xi$ and $\eta$ are vector fields,
$K\in\Gamma(T_{\mathbb{C}}M)$ be the section of $T_{\mathbb{C}}M$.
Let $X$ is the connected component of $Zero(K)$, and $\mathcal{N}_{X}$ be the normal bundle of the connected component $X$,
then
$$\chi(M)=\frac{1}{(-2)^n}\sum_{X}\langle\frac{{\rm ch}(\Lambda_{+}(T^{*}M\otimes\mathbb{C}))-{\rm ch}(\Lambda_{-}(T^{*}M\otimes\mathbb{C}))}{e(\mathcal{N}_{X})},[X]\rangle$$
\end{Theorem}

\section{The difference of the Chern character numbers}
Let $E_{+}$, $E_{-}$ be two complex vector bundles over $M$, and $E=E_{+}\oplus E_{-}$ be the $Z_{2}$-graded complex vector bundle over $M$. Let
$\nabla^{E_{+}}$,$\nabla^{E_{-}}$ be the connection about $E_{+}$ and $E_{-}$, and
$\nabla^{E}=\left(
 \begin{array}{cc}
   \nabla^{E_{+}}& 0 \\
     0 & \nabla^{E_{-}} \\
      \end{array}
 \right)$
be the $Z_{2}$-graded connection on E.

Let $$v\in\Gamma(\textrm{Hom}(E_{+}, E_{-}))$$ be a homomorphism between $E_{+}$ and $E_{-}$.
Let $$v^{*}\in\Gamma(\textrm{Hom}(E_{-}, E_{+}))$$ be the adjoint of $v$ with respect to the Hermitian metrics on $E_{\pm}$ respectively. And
$$V=\left(
 \begin{array}{cc}
   0& v^{*} \\
     v & 0 \\
      \end{array}
 \right)\in\Gamma(\textrm{Hom}(E, E))$$

Let $Z(v)$ denote the set of the points at which $v$ is noninvertible.We always assume that $Z(v)$ is the compact submanifold of $M$, the connected components of $Z(v)$ is denoted by $X$, and $Z(v)=\bigcup X$. Let $U_{X}$ be the tubular neighborhood of the connected component $X$. Let $\mathcal{N}_{X}$ be the normal bundle of the connected component $X$. By the tubular neighborhood theorem(cf.[7]),$U_{X}$ is diffeomorphic to the total space of the normal bundle $\mathcal{N}_{X}$.

\begin{Lemma}The following identity holds,
$$\langle{\rm ch}(E_{+})-{\rm ch}(E_{-}),[M]\rangle=\sum_{X}\langle{\rm ch}(E_{+})-{\rm ch}(E_{-}),[U_{X}]\rangle$$
\end{Lemma}
\begin{proof}
\begin{eqnarray*}
\int_{M}({\rm ch}(E_{+})-{\rm ch}(E_{-}))
&=&\int_{M\setminus\bigcup U_{X}}({\rm ch}(E_{+})-{\rm ch}(E_{-}))+\int_{\bigcup U_{X}}({\rm ch}(E_{+})-{\rm ch}(E_{-}))\\
&=&\int_{M\setminus\bigcup U_{X}}({\rm ch}(E_{+})-{\rm ch}(E_{-}))+\sum_{X}\int_{U_{X}}({\rm ch}(E_{+})-{\rm ch}(E_{-}))
\end{eqnarray*}

Because ${\rm ch}E={\rm ch}(E_{+})-{\rm ch}(E_{-})$ is independent of the choice of the connection $\nabla^{E}$, we need to construct a special connection on $E$(cf.[6]). By
$$[\nabla^{E},V]=\nabla^{E}\cdot V-V\cdot\nabla^{E}=\left(
 \begin{array}{cc}
   0& \nabla^{E_{+}}v^{*}-v^{*}\nabla^{E_{-}} \\
     \nabla^{E_{-}}v-v\nabla^{E_{+}} & 0 \\
      \end{array}
 \right)$$

\begin{eqnarray*}
V[\nabla^{E},V]&=&\left(
 \begin{array}{cc}
   0& v^{*} \\
     v & 0 \\
      \end{array}
 \right)\left(
 \begin{array}{cc}
   0& \nabla^{E_{+}}v^{*}-v^{*}\nabla^{E_{-}} \\
     \nabla^{E_{-}}v-v\nabla^{E_{+}} & 0 \\
      \end{array}
 \right)\\
&=&\left(
 \begin{array}{cc}
   v^{*}(\nabla^{E_{-}}v-v\nabla^{E_{+}})& 0 \\
     0 & v(\nabla^{E_{+}}v^{*}-v^{*}\nabla^{E_{-}}) \\
      \end{array}
 \right)
\end{eqnarray*}
then we can construct two connection on $E$.

$$\nabla_{1}^{E}=\left(
 \begin{array}{cc}
   \nabla^{E_{+}}& 0 \\
     0 & \nabla^{E_{-}} \\
      \end{array}
 \right)+\left(
 \begin{array}{cc}
   v^{*}(\nabla^{E_{-}}v-v\nabla^{E_{+}})& 0 \\
     0 & 0 \\
      \end{array}
 \right),$$

 $$\nabla_{2}^{E}=\left(
 \begin{array}{cc}
   \nabla^{E_{+}}& 0 \\
     0 & \nabla^{E_{-}} \\
      \end{array}
 \right)+\left(
 \begin{array}{cc}
   0& 0 \\
     0 & v(\nabla^{E_{+}}v^{*}-v^{*}\nabla^{E_{-}}) \\
      \end{array}
 \right).$$
Here we only use $\nabla_{1}^{E}$. Because $v$ is invertible on $M\setminus Z(v)$, so we can choose $v^{*}=v^{-1}$ on $M\setminus Z(v)$. And use $\nabla_{1}^{E}$ to construct a new connection
$$\widetilde{\nabla}_{1}^{E}=\left(
 \begin{array}{cc}
   \nabla^{E_{+}}& 0 \\
     0 & \nabla^{E_{-}} \\
      \end{array}
 \right)+\left(
 \begin{array}{cc}
   \rho v^{-1}(\nabla^{E_{-}}v-v\nabla^{E_{+}})& 0 \\
     0 & 0 \\
      \end{array}
 \right)$$
where $\rho$ is a truncating function with $\rho(x)=1, x\in M\setminus\bigcup U_{X}$ and $\rho(x)=0, x\in X$.

So on $M\setminus\bigcup U_{X}$,
$$\widetilde{\nabla}_{1}^{E}=\left(
 \begin{array}{cc}
    v^{-1}\nabla^{E_{-}}v & 0 \\
     0 & \nabla^{E_{-}} \\
      \end{array}
 \right),$$
$$\widetilde{R}_{1}^{E}=(\widetilde{\nabla}_{1}^{E})^{2}=\left(
 \begin{array}{cc}
    v^{-1}R^{E_{-}}v & 0 \\
     0 & R^{E_{-}} \\
      \end{array}
 \right).$$
By the definition of Chern character form (cf.[12], [14] or [1])
$${\rm ch}(E,\widetilde{\nabla}_{1}^{E})={\rm tr_{s}}\left[\exp\left(\frac{\sqrt{-1}}{2\pi}\widetilde{R}_{1}^{E}\right)\right]={\rm tr}\left[\exp\left(\frac{\sqrt{-1}}{2\pi}v^{-1}R^{E_{-}}v\right)\right]-{\rm tr}\left[\exp\left(\frac{\sqrt{-1}}{2\pi}R^{E_{-}}\right)\right]$$
then
$$\int_{M\setminus\bigcup U_{X}}({\rm ch}(E_{+})-{\rm ch}(E_{-}))=\int_{M\setminus\bigcup U_{X}}{\rm ch}(E,\widetilde{\nabla}_{1}^{E})=0.$$
So
$$\int_{M}({\rm ch}(E_{+})-{\rm ch}(E_{-}))=\sum_{X}\int_{U_{X}}({\rm ch}(E_{+})-{\rm ch}(E_{-})).$$

\end{proof}

\begin{Lemma}The following identity holds,
$$\langle{\rm ch}(E_{+})-{\rm ch}(E_{-}),[U_{X}]\rangle=\langle\frac{{\rm ch}(E_{+})-{\rm ch}(E_{-})}{e(\mathcal{N}_{X})},[X]\rangle$$
\end{Lemma}
\begin{proof}
Let $\mathcal{N}_{X}$ be the normal bundle over $X$, consider the maps
$\pi: \mathcal{N}_{X}\rightarrow X$ and $i: X\rightarrow \mathcal{N}_{X}$ where $\pi$ is the bundle projection and $i$ denotes inclusion as the zero section. Let $\pi_{!}$ be the integration over the fibre, and $i_{!}$ be the Thom isomorphism of $\mathcal{N}_{X}$(cf.[8],chapter III.\S 12.).
By assumption $X$ is compact, then we know(cf.[8],chapter III. lemma 12.2.)
 $$i^{*}i_{!}(u)=e(\mathcal{N}_{X})\cdot u$$
for all $u\in H^{*}(X)=H^{*}_{cpt}(X)$, where $e(\mathcal{N}_{X})$ is the Euler class of $\mathcal{N}_{X}$.
If $u=\pi_{!} [({\rm ch}(E_{+})-{\rm ch}(E_{-}))\mid_{\mathcal{N}_{X}}]$, then
$$i^{*}i_{!}(\pi_{!}[({\rm ch}(E_{+})-{\rm ch}(E_{-}))\mid_{\mathcal{N}_{X}}])=e(\mathcal{N}_{X})\cdot \pi_{!} [({\rm ch}(E_{+})-{\rm ch}(E_{-}))\mid_{\mathcal{N}_{X}}]$$
so $$\pi_{!} [({\rm ch}(E_{+})-{\rm ch}(E_{-}))\mid_{\mathcal{N}_{X}}]=\frac{i^{*}i_{!}(\pi_{!}[({\rm ch}(E_{+})-{\rm ch}(E_{-}))\mid_{\mathcal{N}_{X}}])}{e(\mathcal{N}_{X})},$$
because $$i^{*}i_{!}(\pi_{!}[({\rm ch}(E_{+})-{\rm ch}(E_{-}))\mid_{\mathcal{N}_{X}}])=i^{*}[{\rm ch}(E_{+}-E_{-})\mid_{\mathcal{N}_{X}}]={\rm ch}i^{*}[(E_{+}-E_{-})\mid_{\mathcal{N}_{X}}]={\rm ch}(E_{+}-E_{-})\mid_{X}.$$

So $$\pi_{!} [({\rm ch}(E_{+})-{\rm ch}(E_{-}))\mid_{\mathcal{N}_{X}}]=\frac{{\rm ch}(E_{+}-E_{-})\mid_{X}}{e(\mathcal{N}_{X})}=\frac{({\rm ch}(E_{+})-{\rm ch}(E_{-}))\mid_{X}}{e(\mathcal{N}_{X})},$$
$$\langle\pi_{!} [({\rm ch}(E_{+})-{\rm ch}(E_{-}))\mid_{\mathcal{N}_{X}}],[X]\rangle=\langle\frac{{\rm ch}(E_{+})-{\rm ch}(E_{-})}{e(\mathcal{N}_{X})},[X]\rangle$$
then $$\langle{\rm ch}(E_{+})-{\rm ch}(E_{-}),[U_{X}]\rangle=\langle{\rm ch}(E_{+})-{\rm ch}(E_{-}),[\mathcal{N}_{X}]\rangle=\langle\frac{{\rm ch}(E_{+})-{\rm ch}(E_{-})}{e(\mathcal{N}_{X})},[X]\rangle$$

\end{proof}

\section{The proof of Theorem 1.}

By Lemma 1. and Lemma 2. we get the result in Theorem 1.,
$$\langle{\rm ch}(E_{+})-{\rm ch}(E_{-}),[M]\rangle=\sum_{X}\langle\frac{{\rm ch}(E_{+})-{\rm ch}(E_{-})}{e(\mathcal{N}_{X})},[X]\rangle$$

\begin{corollery}[Huitao Feng, Weiping Li and Weiping Zhang]
$$\langle{\rm ch}(E_{+})-{\rm ch}(E_{-}),[M]\rangle=(-1)^{n-1}\sum_{p}\deg(v_{p})$$
\end{corollery}
\begin{proof}
By Theorem 1., if $X=p$ is the isolated zero points, then $\frac{{\rm ch}(E_{+})-{\rm ch}(E_{-})}{e(\mathcal{N}_{p})}=\frac{0}{0}$.
By Lemma 2.
$$\frac{{\rm ch}(E_{+})-{\rm ch}(E_{-})}{e(\mathcal{N}_{p})}=\langle{\rm ch}(E_{+})-{\rm ch}(E_{-}),[U_{p}]\rangle.$$
Let $\nabla^{E}_{t}=(1-t)\nabla^{E}+t\widetilde{\nabla}_{1}^{E}$, by transgression formula
\begin{eqnarray*}
\langle{\rm ch}(E_{+})-{\rm ch}(E_{-}),[U_{p}]\rangle
&=&-\frac{\sqrt{-1}}{2\pi}\int_{U_{p}}d\int^{1}_{0}{\rm tr_{s}}\left[\frac{d\nabla^{E}_{t}}{dt}\exp(\frac{\sqrt{-1}}{2\pi}R_{t}^{E})\right]dt \\
&=&-\frac{\sqrt{-1}}{2\pi}\int_{\partial U_{p}}\int^{1}_{0}{\rm tr_{s}}\left[\frac{d\nabla^{E}_{t}}{dt}\exp(\frac{\sqrt{-1}}{2\pi}R_{t}^{E})\right]dt
\end{eqnarray*}
because we can choose $\nabla^{E}=\left(
 \begin{array}{cc}
   d & 0 \\
     0 & d \\
      \end{array}
 \right)$, then
 $$\nabla^{E}_{t}=(1-t)\left(
 \begin{array}{cc}
   d & 0 \\
     0 & d \\
      \end{array} \right)+t\left[\left(
 \begin{array}{cc}
   d & 0 \\
     0 & d \\
      \end{array} \right)+\left(
 \begin{array}{cc}
   v^{-1}(dv) & 0 \\
     0 & 0 \\
      \end{array} \right)\right]$$

so
$$\frac{\sqrt{-1}}{2\pi}\int_{\partial U_{p}}\int^{1}_{0}{\rm tr_{s}}\left[\frac{d\nabla^{E}_{t}}{dt}\exp(\frac{\sqrt{-1}}{2\pi}R_{t}^{E})\right]dt$$
$$=\frac{\sqrt{-1}}{2\pi}\int_{\partial U_{p}}\int^{1}_{0}{\rm tr}\left[v^{-1}(dv)\frac{1}{(n-1)!}\left(\frac{\sqrt{-1}}{2\pi}t(1-t)(v^{-1}(dv))^{2}\right)^{n-1}\right]dt$$
$$=(\frac{\sqrt{-1}}{2\pi})^{n}\int^{1}_{0}\frac{t^{n-1}(1-t)^{n-1}}{(n-1)!}dt\int_{\partial U_{p}}{\rm tr}\left((v^{-1}(dv))^{2n-1}\right)$$
$$=(\frac{\sqrt{-1}}{2\pi})^{n}\frac{(n-1)!}{(2n-1)!}\int_{\partial U_{p}}{\rm tr}\left((v^{-1}(dv))^{2n-1}\right)$$

Then we get $$\langle{\rm ch}(E_{+})-{\rm ch}(E_{-}),[U_{p}]\rangle=-(\frac{\sqrt{-1}}{2\pi})^{n}\frac{(n-1)!}{(2n-1)!}\int_{\partial U_{p}}{\rm tr}\left((v^{-1}(dv))^{2n-1}\right)=(-1)^{n-1}\deg(v_{p})$$

\end{proof}

\section{The proof of Theorem 2.}
Let $M$ be a closed, oriented smooth manifold of dimension $2n$, $E$
be a oriented real vector bundle on $M$ with rank $2n$. Let
$E_{\mathbb{C}}=E\otimes\mathbb{C}$ denote the complexification of
the vector bundle $E$. Let any $K\in\Gamma(E_{\mathbb{C}})$ be the
section of $E_{\mathbb{C}}$, then $K=\xi+\sqrt{-1}\eta$, where $\xi$
and $\eta$ be smooth section of $E$. Let $g^{E}$ be a Euclidian
inner product on $E$, then it induces canonically a complex
symmetric bilinear form $h^{E_{\mathbb{C}}}$ on $E_{\mathbb{C}}$,
such that
$$h^{E_{\mathbb{C}}}(K,K)=|\xi|^2_{g^{E}}-|\eta|^2_{g^{E}}+2\sqrt{-1}\langle\xi,\eta\rangle_{g^{E}}.$$
The zero points of the smooth function $h^{E_{\mathbb{C}}}(K,K)$ is
denoted by $Zero(K)$.

Let $E^{*}$ be the dual bundle of $E$, set
$\Lambda(E^{*}\otimes\mathbb{C})$ be the exterior algebra bundle
with complex valued. For any $e\in \Gamma(E)$, Clifford element
$c(e)$ acting on $\Lambda(E^{*}\otimes\mathbb{C})$ is defined by
$c(e)=e^*\wedge-i_e$, where $e^{*}$ corresponds to $e$ via $g^{E}$,
$e^{*}\wedge$ and $i_{e}$ are the standard notation for exterior and
interior multiplications.

Let $e_1, e_2,\cdots,e_{2n}$ be the local orthonormal basis of $E$,
set
$$\tau=(\sqrt{-1})^{n}c(e_1)c(e_2)\cdots c(e_{2n})$$
we known that $\tau^{2}=1$ and $\tau$ does not depend on the choice
of the orthonormal basis. Then $\tau$ is a bundle homomorphism on
$\Lambda(E^{*}\otimes\mathbb{C})$, it give the
$\mathbb{Z}_{2}$-grading on $\Lambda(E^{*}\otimes\mathbb{C})$,
$$\Lambda(E^{*}\otimes\mathbb{C})=\Lambda_{+}(E^{*}\otimes\mathbb{C})\oplus\Lambda_{-}(E^{*}\otimes\mathbb{C})$$
where $\Lambda_\pm(E^{*}\otimes\mathbb{C})$ is corresponds to the
characteristic subbundle with characteristic value $\pm$ of the
operator $\tau$. So $\Lambda(E^{*}\otimes\mathbb{C})$ is a super
vector bundle. The $\mathbb{Z}_{2}$-grading is called Signature
$\mathbb{Z}_2$-graded.

For any $e\in\Gamma(E)$, we have $c(e)\tau=-\tau c(e)$, so $c(e)$ is
a bundle homomorphism from $\Lambda_\pm(E^{*}\otimes\mathbb{C})$ to
$\Lambda_\mp(E^{*}\otimes\mathbb{C})$. Then for any $\xi,\eta\in
\Gamma(E)$, we can construct a bundle homomorphism
$$v_{K}=\tau c(\xi)+\sqrt{-1}c(\eta):\Lambda_{+}(E^{*}\otimes\mathbb{C})\rightarrow\Lambda_{-}(E^{*}\otimes\mathbb{C}).$$
Let $v_{K}$ extend to an endomorphism of
$\Lambda(E^{*}\otimes\mathbb{C})$ by acting as zero on
$\Lambda_{-}(E^{*}\otimes\mathbb{C})$, with the notation unchanged.
Let $v^{*}_{K}$ be the adjoint of $v_{K}$ with respect to the
metrics on $\Lambda_{\pm}(E^{*}\otimes\mathbb{C})$ respectively. Set
$V=v_{K}+v^{*}_{K}$. Then $V$ is an odd endomorphism of
$\Lambda(E^{*}\otimes\mathbb{C})$. We use $Z(v_{K})$ to denoted the
noninvertible points of $v_{K}$. $V^{2}$ is fiberwise positive over
$M \setminus Z(v_{K})$ (cf. [4]).
\par

\begin{Lemma}
Let $M$ be a closed, oriented smooth manifold of dimension $2n$,
\begin{description}
\item[1)]If $n\geq2$, then $Z(v_{K})=Zero(K).$
\item[2)]If $n=1$, then $Z(v_{K})=Zero(K)\backslash Z_{+}$\\
where $Z_{+}=\{p\in Zero(K)|\xi(p),\eta(p)$ form a oriented frame on
$E_p$$\}$.
\end{description}
\end{Lemma}
\begin{proof}
Please see [2] or [4].
\end{proof}

We always assume that $Z(v_{K})$ is the compact submanifold of $M$, the
connected components of $Z(v_{K})$ is denoted by $X$.

\begin{Lemma}The following identity holds,
$$\langle{\rm ch}(\Lambda_{+}(E^{*}\otimes\mathbb{C}))-{\rm ch}(\Lambda_{-}(E^{*}\otimes\mathbb{C})),[M]\rangle=(-2)^{n}\chi(E)$$
\end{Lemma}
\begin{proof}
This is a well known result, Please see [2] for a proof from differential geometry.
\end{proof}

\begin{corollery}
$$\langle{\rm ch}(\Lambda_{+}(T^{*}M\otimes\mathbb{C}))-{\rm ch}(\Lambda_{-}(T^{*}M\otimes\mathbb{C})),[M]\rangle=(-2)^n\chi(M)$$
\end{corollery}
\begin{proof}
By Lemma 4., if $E=TM$ so we get the result.
\end{proof}

Now we can give the proof of the Theorem 2.
By Theorem 1. and Corollery 2., we have
\begin{eqnarray*}
(-2)^n\chi(M)
&=&\langle{\rm ch}(\Lambda_{+}(T^{*}M\otimes\mathbb{C}))-{\rm ch}(\Lambda_{-}(T^{*}M\otimes\mathbb{C})),[M]\rangle \\
&=&\sum_{X}\langle\frac{{\rm ch}(\Lambda_{+}(T^{*}M\otimes\mathbb{C}))-{\rm ch}(\Lambda_{-}(T^{*}M\otimes\mathbb{C}))}{e(\mathcal{N}_{X})},[X]\rangle
\end{eqnarray*}
So
$$\chi(M)=\frac{1}{(-2)^n}\sum_{X}\langle\frac{{\rm ch}(\Lambda_{+}(T^{*}M\otimes\mathbb{C}))-{\rm ch}(\Lambda_{-}(T^{*}M\otimes\mathbb{C}))}{e(\mathcal{N}_{X})},[X]\rangle.$$

\section{Conclusion}
We give an answer to H. Jacobowitz's question for $Zero(K)$ is a submanifold. The answer is found to be related to the Poincar\'e-Hopf index formula. The classical Poincar\'e-Hopf index formula is the relation about the Euler charateristic number $\chi(M)$ of manifold $M$ and the degree of vector field at the zeros(singularities). There are many ways to understand the Poincar\'e-Hopf index formula (cf.[15],[3],[5],[11],[13]). The difference of the Chern character numbers of two complex vector bundles give the understanding of the Poincar\'e-Hopf index formula by a new way. This way is established in [4] with isolated singularities. We give the way with non-isolated singularities. This Poincar\'e-Hopf type formula also can be used to express the index of a special twisted Dirac operator(cf.[9]).

\end{document}